\UseRawInputEncoding
\documentclass[12pt]
{amsart}
\usepackage[english]{babel}
\usepackage{amsfonts,amsmath,amssymb,amsthm,graphicx,afterpage,url}
\input{xy}
\xyoption{all}
\usepackage{epsfig}
\usepackage{overpic}
\usepackage{pgf,tikz}
\usetikzlibrary{arrows}

\graphicspath{{morefig/}{figures/}{figures00/}{figures02/}{figures05/}{figlms/}{figumisha/}{knots-figures/}}

\newcommand\aronly[1]{#1}
\newcommand\jonly[1]{}
\newcommand\algor[1]{}
\newcommand\invadraw[1]{}

\def\lk{\mbox{lk}}

\def\Cl{\mbox{Cl}}
\DeclareMathOperator{\pro}{pr}
\DeclareMathOperator{\inc}{\lambda}
\DeclareMathOperator{\incj}{\mu}
\def\Cyl{\mathop{\fam0 Cyl}}
\def\ret{\mathop{\fam0 ret}}

\def\id{\mathop{\fam0 id}}

\def\sgn{\mathop{\fam0 sgn}}

\def\lk{\mathop{\fam0 lk}}

\def\Cl{\mathop{\fam0 Cl}}

\def\R{{\mathbb R}} \def\Z{{\mathbb Z}} \def\C{\Bbb C} 
\let\Bbb=\mathbb

\long\def\comment#1\endcomment{}

\theoremstyle{plain}
\newtheorem{theorem}{Theorem}[section]
\newtheorem{lemma}[theorem]{Lemma}

\newtheorem{proposition}[theorem]{Proposition}
\newtheorem{conjecture}[theorem]{Conjecture}

\newtheorem{example}[theorem]{Example}

\theoremstyle{definition}
\newtheorem{remark}[theorem]{Remark}

\begin{document}






\title{Extendability of simplicial maps is undecidable}

\author{A. Skopenkov}

\thanks{I would like to thank M. \v Cadek, R. Karasev, \aronly{E. Kogan, B. Poonen,} L. Vok\v r\'inek, U. Wagner, and the anonymous referee for helpful discussions.
\newline
Moscow Institute of Physics and Technology,
and Independent University of Moscow.
Email: \texttt{skopenko@mccme.ru}.
\texttt{https://users.mccme.ru/skopenko/}.
\aronly{Supported by the Russian Foundation for Basic Research Grant No. 19-01-00169.}}

\date{}

\maketitle

\begin{abstract}
We present a short proof of the \v{C}adek-Kr\v{c}\'al-Matou\v{s}ek-Vok\v{r}\'inek-Wagner result from the title (in the following form due to Filakovsk\'y-Wagner-Zhechev).

{\it For any fixed even $l$ there is no algorithm recognizing the extendability of the identity map of $S^l$ to a PL map $X\to S^l$
of given $2l$-dimensional simplicial complex $X$ containing a subdivision of $S^l$ as a given subcomplex.}

We also exhibit a gap in the Filakovsk\'y-Wagner-Zhechev proof that embeddability of complexes is undecidable in codimension $>1$.
\end{abstract}

\aronly{\tableofcontents}

\comment

55-02, 55P05, 55S36, 68-02, 68U05



Хорошо известно, что существует алгоритм распознавания планарности графа, линейный по количеству его вершин
(Хопкрофт-Тарджан, 1974).
Мы рассмотрим аналогичную задачу для гиперграфов в пространствах произвольной размерности:
как распознать  вложимость  $k$-мерного (т.е. $(k+1)$-однородного) гиперграфа в $d$-мерное пространство?

Для коразмерности $d-k=0,1$ указанная в названии нераспознаваемость  несложно вытекает из теоремы Новикова о нераспознаваемости $k$-мерной сферы для $k>4$ (это показали Matou\v sek-Tancer-Wagner, arXiv:0807.0336 [cs.CG]).
Я расскажу о соответствующем результате для коразмерности $d-k>2$, который  анонсирован в 2019 (Filakovsk\'y-Wagner-Zhechev). Его красивое доказательство использует алгоритмическую нераспознаваемость продолжаемости отображений (\v Cadek-Kr\v cal-Matou\v sek- Vokrinek-Wagner, arXiv:1302.2370 [cs.CG]).
Я приведу простой вывод последнего из нераспознаваемости разрешимости диофантовых уравнений и топологической теоремы Брауэра-Хопфа-Уайтхеда.

\endcomment

\section{Extendability of simplicial maps is undecidable}\label{s:mainre}

We present short proofs of recent topological undecidability results for hypergraphs (complexes): Theorems \ref{t:hopfwhs} and \ref{t:exun} \cite{CKM+, FWZ}.



A {\bf complex} $K=(V,F)$ is a finite set $V$ together with a collection $F$ of subsets of $V$ such that if a subset $\sigma$ is in $F$, then every subset of $\sigma$ is in $F$.\footnote{We do not use longer name `abstract finite simplicial complex'.
A {\it $k$-hypergraph} (more precisely, a $(k+1)$-uniform hypergraph) $(V,F)$ is a finite set $V$ together with a collection $F$ of $(k+1)$-element subsets of $V$.
In topology it is more traditional (because often more convenient) to work with complexes not hypergraphs.
The following results are stated for complexes, although some of them are correct for hypergraphs.}
(Hence $F\ni\emptyset$.)
In an equivalent geometric language, a complex is a collection of closed faces (=subsimplices) of some simplex.
A {\bf $k$-complex} is a complex containing at most $(k+1)$-element subsets, i.e., at most $k$-dimensional simplices.
Elements of $V$ and of $F$ are called {\bf vertices} and {\bf faces}.

{\it The complete $k$-complex on $n$ vertices} (or the $k$-skeleton of the $(n-1)$-simplex)
is the collection of all at most $(k+1)$-element subsets of an $n$-element set.
For $k=0$ we denote this complex by $[n]$, for $n=k+1$ by $D^k$ ($k$-simplex or $k$-disk), and for $n=k+2$ by $S^k$ ($k$-sphere).

\begin{figure}[h]\centering
\includegraphics[scale=0.9]{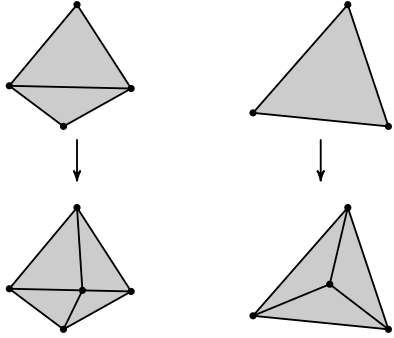}
\caption{Subdivision of an edge (left) and of a face (right)}
\label{podra1}
\end{figure}

The {\it subdivision of an edge} operation is shown in fig.~\ref{podra1}. 
Exercise: represent
the {\it subdivision of a face} operation shown in fig.~\ref{podra1} as composition of several subdivisions of an edge and inverse operations).
A {\bf subdivision} of a complex $K$ is any complex 
obtained from $K$ by several subdivisions of edges.

A {\bf simplicial map}   $f:(V,F)\to(V',F')$ between complexes is a map $f:V\to V'$ (not necessarily injective)
such that $f(\sigma)\in F'$ for each $\sigma\in F$.
A {\bf piecewise-linear (PL) map} $K\to K'$ between complexes is a simplicial map between certain their subdivisions.

The {\bf body} (or geometric realization) $|K|$ of a complex $K$ is the union of simplices of $K$.
Below we often abbreviate $|K|$ to $K$; no confusion should arise.
A  simplicial or PL map between complexes induces a map between their bodies, which is called {\it simplicial} or {\it PL}, respectively.\footnote{The related different notion of a {\it continuous} map between bodies of complexes is not required to state and prove the results of this text.
In theorems below the existence of a continuous extension is equivalent to the existence of a PL extension (by the PL  Approximation Theorem).}

The {\bf wedge} $K_1\vee\ldots\vee K_m$ of complexes $K_1=(V_1,F_1),\ldots,K_m=(V_m,F_m)$ with disjoint vertices is the complex whose vertex set is obtained by choosing one vertex from each $V_j$ and identifying chosen vertices, and whose of faces is obtained from $F_1\sqcup\ldots\sqcup F_m$ by such identification.
The choice of vertices is important in general, but is immaterial in the examples below.
Let $K\vee K$ be the wedge of two copies of $K$.

Let $Y_l=S^l$ for $l$ even and $Y_l=S^l\vee S^l$ for $l$ odd.

\begin{theorem}[retractability is undecidable]\label{t:hopfwhs}
For any fixed integer $l>1$ there is no algorithm recognizing the extendability of the identity map of $Y_l$
to a PL map $X\to Y_l$ of given $2l$-complex $X$ containing a subdivision of $Y_l$ as a given subcomplex.
\end{theorem}

This is implied \cite{FWZ} by the following theorem and Proposition \ref{p:cylnew}.b.

Let $V^d_m=S^d_1\vee\ldots\vee S^d_m$ be the wedge of $m$ copies of $S^d$.

\begin{theorem}[extendability is undecidable]\label{t:exun}
For some fixed integer $m$ and any fixed integer $l>1$ there is no algorithm recognizing extendability of given simplicial map $V^{2l-1}_m\to Y_l$ to a PL map $X\to Y_l$ of given $2l$-complex $X$ containing a subdivision of $V^{2l-1}_m$ as a given subcomplex.
\end{theorem}

This is a `concrete' version of \cite[Theorem 1.1.a]{CKM+}.

Remarks and examples below are formally not used later.

\begin{remark}\label{r:ear}
(a) {\it Relation to earlier known results.} For $l>1$ any PL map $S^1\to Y_l$ extends to $D^2$.
The analogues of Theorems \ref{t:hopfwhs} and \ref{t:exun} for $Y_l$ replaced by a complex without this property (called {\it simply-connectedness}) were well-known by mid 20th century.
See more in \cite[\S1]{CKM+}.

(b) {\it Why this text might be interesting.}
Exposition of the proofs of Theorems \ref{t:hopfwhs} and \ref{t:exun} here is shorter and simpler than in  \cite{CKM+}.
I structure the proof by explicitly stating
the Brower-Hopf-Whitehead Theorems \ref{t:brow}, \ref{t:himi}, and Propositions \ref{p:ld}, \ref{p:cylnew}.
Theorems \ref{t:brow} and \ref{t:himi} relate homotopy classification to quadratic functions on integers.
Thus they allow to prove the equivalence of extendability / retractability to homotopy of certain maps, and to solvability of certain Diophantine equations, see Propositions \ref{p:ld} and \ref{p:cylnew}.
These results are essentially known before \cite{CKM+} and are essentially deduced in \cite{CKM+}
from other known results.
(As far as I know, they were not explicitly stated earlier, not even in \cite{CKM+}; cf. \cite[\S4.2]{CKM+} and \cite[Remarks 2.1.b and 2.2.e]{Sk21d}.)

Also I present definitions in an economic way accessible to non-specialists (including computer scientists).
In particular, I
do not use cell complexes and simplicial sets.

A reader might want to consider the proof below first for $l$ even.
Then he/she can omit parts (b,b') of Lemma \ref{l:solv}, part (c) of Theorem \ref{t:brow}, and parts (b1,b2) of Theorem \ref{t:himi}.
\end{remark}


\begin{lemma}\label{l:solv}
(a) For some (fixed) integers $m,s$ there is no algorithm that for given arrays $a=((a^{i,j})_1,\ldots,(a^{i,j})_m)$, $1\le i<j\le s$, and $b=(b_1,\ldots,b_m)$ of integers decides whether

\quad(SYM) there are integers $x_1,\ldots,x_s$ such that
$$\sum\limits_{1\le i<j\le s}a^{i,j}_qx_ix_j=b_q\quad\text{for any}\quad 1\le q\le m.$$

(b) Same as (a) for

\quad(SKEW) there are integers $x_1,\ldots,x_s,y_1,\ldots,y_s$ such that
$$\sum\limits_{1\le i<j\le s}a^{i,j}_q(x_iy_j-x_jy_i)=b_q\quad\text{for any}\quad 1\le q\le m.$$

(b') Same as (a) for the property (SKEW') obtained from (SKEW) by replacing $a^{i,j}_q$ with $2a^{i,j}_q$.
\end{lemma}

See \cite[\S2]{CKM+} for deduction of (a,b) from insolvability of general Diophantine equations.
Part (b') follows by (b) because either all $b_q$ in (SKEW') are even or the system (SKEW') is unsolvable.

Denote by $\simeq$ homotopy between maps.
For $n>1$ we use {\it abelian group structure} on the set $\pi_n(X)$ of homotopy classes of PL maps $S^n\to X$.
Let $u,v:S^l\vee S^l\to S^l$ be the contractions of the second and the first sphere of $S^l\vee S^l$.

\begin{theorem}
\label{t:brow}
For any integer $l$ and simplicial map $\varphi:P\to Q$ between subdivisions of $S^l$ there is an effectively constructible
integer $\deg\varphi$ (called the \emph{degree} of $\varphi$) such that

(a) for any integer $k$ there is an effectively constructible PL map $\widehat k:S^l\to S^l$ of degree $k$;

(b) for maps $\varphi,\psi:S^l\to S^l$ if $\deg\varphi=\deg\psi$, then $\varphi\simeq\psi$.

(c) for $l>1$ and maps $\varphi,\psi:S^l\to S^l\vee S^l$ if $\deg(u\circ\varphi)=\deg(u\circ\psi)$ and
$\deg(v\circ\varphi)=\deg(v\circ\psi)$, then $\varphi\simeq\psi$.
\end{theorem}

\begin{proof}[Sketch of a proof]
Define $\deg\varphi$ by to be the sum of signs of a finite number of points from $\varphi^{-1}y$,
where $y\in S^l$ is a `random' (i.e. {\it regular}) value of $\varphi$.
More precisely, take $y$ outside the image of any $(l-1)$-simplex of $P$.
For the definition of sign and the proof of (b) see e.g. \cite{Ma03} or \cite[\S8]{Sk20}.
Part (c) is a simple case of the Hilton Theorem.

Clearly, $\deg$
defines a homomorphism $\pi_l(S^l)\to\Z$.
Let $\widehat 1:=\id S^l$, let $\widehat 0$ be the constant map, and let
$\widehat{-1}$ be the reflection w.r.t. the equator $S^{l-1}\subset S^l$.
Then for $k\ne0$ let $\widehat k$ be a representative of the sum of $|k|$ summands $\widehat{\sgn k}$.
\end{proof}

For a set $x=(x_1,\ldots,x_s)$ of integers let $\widehat x:V^l_s\to S^l$ be the map whose restriction to $S^l_j$ is $\widehat{x_j}$.
Let $\inc,\incj:S^l\to S^l\vee S^l$ be the inclusions into the first and the second sphere of the wedge.

\begin{theorem}[proved in \S\ref{s:sketch}]\label{t:himi}
For any integer $a$ there exists an effectively constructible PL map $W_2(a):S^{2l-1}\to S^l_1\vee S^l_2$ such that for any $l>1$

(a) for the composition
$W(a):S^{2l-1}\overset{W_2(a)}\to S^l_1\vee S^l_2\overset{\id\vee\id}\to S^l$
we have

\quad (a1) for $l$ even $W(a)\simeq W(a')$ only when $a=a'$;

\quad (a2) $\widehat{(x_1,x_2)}\circ W_2(a)\simeq W(ax_1x_2)$.

(b1) $W_2(a)\simeq W_2(a')$ only when $a=a'$;

(b2) for $l$ odd $(\inc\circ\widehat{(x_1,x_2)}+\incj\circ\widehat{(y_1,y_2)})\circ W_2(2a)\simeq W_2(2a(x_1y_2-x_2y_1))$, where the map
$\inc\circ\widehat x+\incj\circ\widehat y:S^l_1\vee S^l_2\to S^l\vee S^l$ is defined to be
$\inc\circ\widehat x_j+\incj\circ\widehat y_j$ on $S^l_j$.
\end{theorem}

\comment


\begin{theorem}[proved in \S\ref{s:sketch}]\label{t:himi}
For any array $a=(a^{i,j})$, $1\le i<j\le s$, of integers\footnote{For $s=2$ this array consists of one integer.} there exists
an effectively constructible PL map $W_s(a):S^{2l-1}\to V^l_s$ such that

$\bullet$ for any integer $b$, any even $l$, and the composition
$$W(b):S^{2l-1}\overset{W_2(b)}\to S^l\vee S^l\overset{\id\vee\id}\to S^l\quad\text{we have}$$

\quad $W(b)\simeq W(b')$ only when $b=b'$;

\quad $\widehat x\circ W_s(a)\simeq W(Q_x(a))$, where\footnote{For $s=2$ we have $Q_x(a)=ax_1x_2$.}
$Q_x(a):=\sum\limits_{1\le i<j\le s} a^{i,j}x_ix_j$.

$\bullet$ for any odd $l>1$ we have

\quad $W_s(a)\simeq W_s(a')$ only when $a=a'$;

\quad $(\inc\circ\widehat x+\incj\circ\widehat y)\circ W_s(a)\simeq W_2(R_{x,y}(a))$, where all $a^{i,j}$ even, the map
$\inc\circ\widehat x+\incj\circ\widehat y:V^l_s\to S^l\vee S^l$ is defined to be
$\inc\circ\widehat x_j+\incj\circ\widehat y_j$ on $S^l_j$,
and\footnote{For $s=2$ we have $R_{x,y}(a)=a(x_1y_2-x_2y_1)$.}
$R_{x,y}(a):=\sum\limits_{1\le i<j\le s} a^{i,j}(x_iy_j-x_jy_i)$.
\end{theorem}


\endcomment

\begin{proposition}\label{p:ld}
Let $a=((a^{i,j})_1,\ldots,(a^{i,j})_m)$, $1\le i<j\le s$, and $b=(b_1,\ldots,b_m)$ be arrays of integers.
There are effectively constructible PL maps
$$W_s(a):V^{2l-1}_m\to V^l_s\quad\text{and}\quad W(b):V^{2l-1}_m\to S^l$$
such that  the property (SYM) for even $l$, and  the property  (SKEW) for odd $l>1$ and all $a^{i,j}_q$ even,
is equivalent to

(LD) there is a PL map $\varkappa:V^l_s\to Y_l$ such that $\varkappa\circ W_s(a) \simeq \beta_l(b)$.

Here $\beta_l=W$ for $l$ even, and $\beta_l=W_2$ for $l$ odd.
\end{proposition}

Proposition \ref{p:ld} and Lemma \ref{l:solv}.ab' imply that {\it homotopy left divisibility} is undecidable.
See the right triangle of the left diagram below for $P=V^{2l-1}_m$, $Q=V^l_s$, $g=W_s(a)$, and $\beta=\beta_l(b)$.
$$
\xymatrix{ & P \ar[dl]_\subset \ar[d]_g \ar[dr]^\beta & \\
\Cyl g \ar@{-->}@(rd,ld)[rr] & Q \ar[l]_\subset \ar@{-->}[r]^\varkappa & Y}
\qquad
\xymatrix{P \ar[d]_g \ar[r]^{\subset} & X \ar@{-->}[dl] \ar[r]^{\subset}  & X\cup_P\Cyl g \ar@{-->}[dll]  \\
Q \ar@{=}[rr] & & Q \ar[u]_{\subset} }
$$


\begin{proof}[Deduction of Proposition \ref{p:ld} from Theorem \ref{t:himi}.]
Let

$\bullet$ $W_s^{i,j}(a^{i,j}_q)$ be the composition $S^{2l-1}\overset{W_2(a^{i,j}_q)}\to S^l_i\vee S^l_j \overset{\subset}\to V^l_s$;

$\bullet$ $W_s(a_q):S^{2l-1}\to V^l_s$ be any PL map representing the sum of the maps $W_s^{i,j}(a_q)$;

$\bullet$ $W_s(a):V^{2l-1}_m\to V^l_s$ be the map whose restriction to the $q$-th sphere is $W_s(a_q)$;

$\bullet$ $W(b):V^{2l-1}_m\to S^l$ be the map whose restriction to the $q$-th sphere is $W(b_q)$.


By Theorem \ref{t:himi} and using $(\alpha_1+\alpha_2)\circ\gamma = \alpha_1\circ\gamma + \alpha_2\circ\gamma$,
 for $l>1$ we have

(a2s) $\widehat x\circ W_s(a)\simeq W(Q_x(a))$, where $Q_x(a):=\sum\limits_{1\le i<j\le s} a^{i,j}x_ix_j$.

(b1s) $W_s(a)\simeq W_s(a')$ only when $a=a'$;

(b2s) for $l$ odd $(\inc\circ\widehat x+\incj\circ\widehat y)\circ W_s(2a)\simeq W_2(2R_{x,y}(a))$, where
$R_{x,y}(a):=\sum\limits_{1\le i<j\le s} a^{i,j}(x_iy_j-x_jy_i)$.

{\it Proof that $(SYM)\Rightarrow(LD)$ for $l$ even.} Take an integer solution $x=(x_1,\ldots,x_s)$.
Let $\varkappa:=\widehat x$.
Then by (a2s) $\varkappa\circ W_s(a^q)\simeq W(Q_x(a^q))=W(b_q)$ for each $q$.
Thus $\varkappa\circ W_s(a)\simeq W(b)$.


{\it Proof that $(LD)\Rightarrow(SYM)$ for $l$ even.} Take the PL map $\varkappa:V^l_s\to S^l$.
Let $x_j:=\deg(\varkappa|_{S^l_j})$.
Then by Theorem \ref{t:brow}.b $\varkappa\simeq\widehat x$.
Take any $q$.
Then by (a2s)
$$W(Q_x(a^q))\simeq \widehat x\circ W_s(a^q)\simeq \varkappa\circ W_s(a^q)\simeq W(b_q).$$
Hence by (a1) of Theorem \ref{t:himi} $Q_x(a^q)=b_q$.

{\it Proof that $(SKEW)\Rightarrow(LD)$ for $l>1$ odd and all $a^q_{i,j}$ even.}
Take an integer solution $(x,y)=(x_1,\ldots,x_s,y_1,\ldots,y_s)$.
Let $\varkappa:=\inc\circ\widehat x+\incj\circ\widehat y$.
Then by (b2s) $\varkappa\circ W_s(a^q)\simeq W(R_{x,y}(a^q))=W_2(b_q)$ for each $q$.
Thus $\varkappa\circ W_s(a)\simeq W_2(b)$.

{\it Proof that $(LD)\Rightarrow(SKEW)$ for $l>$ odd and all $a^q_{i,j}$ even.}
Let $x_j:=\deg(u\circ\varkappa|_{S^l_j})$ and $y_j:=\deg(v\circ\varkappa|_{S^l_j})$.
Then by Theorem \ref{t:brow}.c $\varkappa\simeq\inc\circ\widehat x+\incj\circ\widehat y$.
Take any $q$.
Then by (b2s)
$$W_2(R_{x,y}(a^q))\simeq(\inc\circ\widehat x+\incj\circ\widehat y)\circ W_s(a^q)\simeq\varkappa\circ W_s(a^q)\simeq W_2(b_q).$$
Hence by (b1s) $R_{x,y}(a^q)=b_q$.
\end{proof}

\begin{proposition}[proved below in \S\ref{s:mainre}]\label{p:cylnew}
For a simplicial map $g:P\to Q$ between complexes there is an effectively constructible triple $(\Cyl g;P,Q)$
of a complex $\Cyl g$ (called the \emph{mapping cylinder} of $g$) and its subcomplexes isomorphic to $P,Q$ such that

(a) for any complex $Y$ a simplicial map $\beta:P\to Y$ extends to $\Cyl g$ if and only if
there is a PL map $\varkappa:Q\to Y$ such that $\varkappa\circ g\simeq\beta$.

(b) $g$ extends to a complex $X\supset P$ if and only if the identity map of $Q$ extends to $X\cup_P\Cyl g$;
\end{proposition}

\begin{proof}[Proof of the `extendability is undecidable' Theorem \ref{t:exun}]
Take a simplicial subdivision of $\beta_l(b)$ as a given map, and $X=\Cyl W_s(a)$.
Apply Propositions \ref{p:ld} and \ref{p:cylnew}.a (in the latter take $P=V^{2l-1}_m$, $Q=V^l_s$, $Y=Y_l$, $g=W_s(a)$, and $\beta=\beta_l(b)$).
We obtain that

$\bullet$ extendability of $\beta_l(b)$ to $X$ is equivalent to (SYM) for $l$ even;

$\bullet$ when all $a^{i,j}_q$ are even, extendability of $\beta_l(b)$ to $X$ is equivalent to (SKEW) for $l$ odd.

The latter is undecidable by Lemma \ref{l:solv}.a,b'.
\end{proof}

{\it Sketch of a construction of $\Cyl g$ in Proposition \ref{p:cylnew}.}
For a map $f:P\to Q$ between subsets $P\subset\R^p$ and $Q\subset\R^q$ define the {\it mapping cylinder} $\Cyl f$ to be the union of
$0\times Q\times 1\subset\R^p\times\R^q\times\R=\R^{p+q+1}$ and segments joining points $(u,0,0)\in\R^{p+q+1}$ to $(0,f(u),1)\in\R^{p+q+1}$,
for all $u\in P$.
See \cite[Figure in p. 14]{CKM+}.
We identify $P$ with $P\times 0\times 0$ and $Q$ with $0\times Q\times 1$.

Define the map $\ret g:\Cyl g\to Q$ by mapping to $g(u)$ the segment containing $(u,0,0)$.

For a simplicial map $g:P\to Q$ between complexes denote by $|g|:|P|\to |Q|$ the corresponding PL map
between their bodies.
Then $\Cyl|g|$ is the body of certain complex

$\bullet$ whose vertices are the vertices of $P$ and the vertices of $Q$;

$\bullet$ whose simplices are the simplices of $P$, the simplices of $Q$ and another simplices that are not hard to define.

\begin{example}\label{e:cyl}
(a) For the 2-winding $\widehat 2:S^1\to S^1$ (i.e., for the quotient map $S^1\to\R P^1$) $\Cyl\widehat 2$  is the M\"obius band (i.e. the complement to a 2-disk in $\R P^2$).

(a') For the \emph{Hopf map} $\eta:S^3\to S^2$ (i.e., for the quotient map $S^3\to\C P^1$)
$\Cyl\eta$ is the complement to a 4-ball in $\C P^2$ (i.e. the `complexified' M\"obius band).

(b) For the \emph{commutator map} $f:S^1\to S^1\vee S^1$ (i.e., $f=aba^{-1}b^{-1}$)
$\Cyl f$ is the complement to a 2-disk in $S^1\times S^1$.

(b') The cylinder of the map $W(1):S^{2l-1}\to S^l\vee S^l$ is the complement to a $2l$-ball in $S^l\times S^l$\aronly{ (this follows by footnote \ref{f:whitd})}.
\end{example}


\begin{proof}[Proof of Proposition \ref{p:cylnew}]
{\it (a), `only if'.} Let $\varkappa$ be the restriction to $Q\subset\Cyl g$ of given extension.

{\it (a), `if'.} Let the required extension be $\varkappa\circ\ret g$.

{\it (b), `only if'.} Let the required extension be $\ret g$ on $\Cyl g$ and the given extension on $X$.

{\it (b), `if'.} Let $r:X\cup_P\Cyl g\to Q$ be given extension.
The composition $P\times[0,1]\to\Cyl g\to Q$ of the quotient map and $r$ is a homotopy between $r|_P$ and $g$.
Since $r|_P$ extends to $X$, by the Borsuk Homotopy Extension Theorem\footnote{\label{f:bhet}This theorem states that if $(K,L)$ is a polyhedral pair, $Q\subset\R^d$, \ $F:L\times I\to Q$ is a homotopy and $g:K\to Q$ is a map such that $g|_L=F|_{L\times0}$, then $F$ extends to a homotopy $G:K\times I\to Q$ such that $g=G|_{K\times0}$.}
 it follows that $g$ extends to $X$.
\end{proof}

\aronly{

\begin{remark}\label{r:gap} Proposition 5.2 of \cite{CKM+} asserts the equivalence of (SKEW) and extendability of $W_2(b)$ to $\Cyl W_s(a)$ (which follows by Propositions \ref{p:ld} and \ref{p:cylnew}.a).
The proof of Proposition 5.2 was not formally presented in \cite{CKM+}, it is written that the proposition follows from the text before.
The phrase `For this system, the above equation is exactly the one from (Q-SKEW)' before
\cite[Proposition 5.2]{CKM+} is incorrect.
Indeed, `the above equation' is an equation in $\pi_{2k-1}(S^k\vee S^k)$ not in $\Z$ (and not in the direct summand $\pi_{2k-1}(S^{2k-1})\cong\Z$ of $\pi_{2k-1}(S^k\vee S^k)$), so `the above equation' is not `exactly the one from (Q-SKEW)'.
So the phrase `We get the following:' before \cite[Proposition 5.2]{CKM+} is not justified.
For its justification one needs to prove that multiplication by 2 of `the system of $s$ equations in $\pi_{2k-1}(S^k \vee S^k)$' produces an equivalent system.
This is not so because the group $\pi_{2k-1}(S^k \vee S^k)$ (at one place denoted by $\pi_{2k-1}(S^d\vee S^d)$) can have elements of order 2 for some $k$.
Thus the `if' part of Proposition 5.2 is not proved in \cite{CKM+}.
This gap is easy to recover; e.g. it is recovered here.\footnote{I am grateful to M. \v Cadek for confirming that \cite[Proposition 5.2]{CKM+} is incorrect but is easily correctible.
If the (minor) gap would be recovered in the arXiv update of \cite{CKM+}, I would be glad to remove this footnote.
\newline
I am also grateful to L. Vok\v r\'inek for the following explanation why Remark \ref{r:gap} is not proper:
{\it There is a misprint in the statement of \cite[Proposition 5.2]{CKM+}; namely, the assumption that all coefficients $a_{ij}^{(q)}$ in \cite[(Q-SKEW)]{CKM+}=(SKEW) should be even is missing. With this assumption in place, I am not aware of any gap.
Remark \ref{r:gap} is incorrect in assuming that \cite{CKM+} uses multiplication by 2 in the homotopy group $\pi_{2k-1}(S^k \vee S^k)$; instead, the system of equations \cite[(Q-SKEW)]{CKM+}=(SKEW) with values in $\Z$ gets multiplied.}
This does not show that Remark \ref{r:gap} is not proper, because Remark \ref{r:gap} concerns only the text \cite{CKM+}, not any other non-existent text, cf. \cite[Remark 2.3.d]{Sk21d}.
The correction suggested by L. Vok\v r\'inek is proper; a list of this and all the induced corrections to
\cite{CKM+} would be helpful (or current lack of such a list is helpful)
to see how proper is to call this a misprint. Cf. \cite[Remark 2.3.abc]{Sk21d}.}
\end{remark}

}




\section{Known proof of  Theorem \ref{t:himi}}\label{s:sketch}


{\it Construction of $W_2(a)$.}
Decompose
$$S^{2l-1}= \partial(D^l\times D^l) =S^{l-1}\times D^l\cup_{S^{l-1}\times S^{l-1}}D^l\times S^{l-1}.$$
Define the {\it Whitehead map} $w:S^{2l-1}\to S^l\vee S^l$ as the `union' of the compositions
$$S^{l-1}\times D^l\overset{\pro_2}\to D^l\overset{c}\to S^l\overset{\inc}\to S^l\vee S^l\quad\text{and}\quad
D^l\times S^{l-1}\overset{\pro_1}\to D^l\overset{c}\to S^l\overset{\incj}\to S^l\vee S^l.$$
Here $\pro_j$ is the projection onto the $j$-the factor, and $c$ is contraction of the boundary to a point.\aronly{\footnote{\label{f:whitd}Observe
that $S^l\times S^l\cong D^{2l}/{\sim}$, where $x\sim y \Leftrightarrow \left(x,y\in S^{2l-1}\text{ and }w(x)=w(y)\right)$.}}
It is easy to modify this `topological' definition to obtain an effectively constructible PL map $w$.
Define $W_2(a)$ in the same way as $w$ except that $\inc$ is replaced with $\inc\circ\widehat a$.

\smallskip

Denote by $\lk$ the {\it linking coefficient} of two collections of oriented closed polygonal lines in $\R^3$,
or, more generally, of two integer $l$-cycles in $\R^{2l-1}$.
See definition e.g. in \cite{ST80}, \cite[\S4]{Sk}.


\begin{proof}[Sketch of a proof of (a1)]
For a PL map $\psi:S^{2l-1}\to S^l$ define $H(\psi):=\lk(\psi^{-1}y_1,\psi^{-1}y_2)$, where
$y_1,y_2\in S^l$ are distinct `random' (or {\it regular}) values of $\psi$, and $\psi^{-1}$ is `oriented' preimage.
More precisely, take subdivisions of $S^{2l-1}$ and of $S^l$ for which $\psi$ is simplicial.
Then take $y_1,y_2$ outside the image of any $(l-1)$-simplex
of the subdivision of $S^{2l-1}$.
This is a well-defined homotopy invariant of $\psi$ ({\it Hopf invariant}).

For $l$ even we have $HW(b)=\pm2b$.\aronly{\footnote{For $l$ odd we have $H(\psi)=0$ for any $\psi$.}}
Hence $W(b)\simeq W(b')$ only when $b=b'$.
\end{proof}

\begin{proof}[Sketch of a proof of (b1)]
For a PL map $\psi:S^{2l-1}\to S^l\vee S^l$ define $H_\vee(\psi):=\lk(\psi^{-1}y_1,\psi^{-1}y_2)$, where
$y_1\in S^l\vee*$, $y_2\in *\vee S^l$ are `random' (or {\it regular}) values of $\psi$, and $\psi^{-1}$ is `oriented' preimage.
More precisely, take subdivisions of $S^{2l-1}$ and of $S^l\vee S^l$ for which $\psi$ is simplicial.
Then take $y_1,y_2$ outside the image of any $(l-1)$-simplex of the subdivision of $S^{2l-1}$.
This is a well-defined homotopy invariant of $\psi$ ({\it Whitehead invariant}).

Clearly, $H_\vee W_2(a)=\pm a$.
Hence $W_2(a)\simeq W_2(a')$ only when $a=a'$.
\end{proof}

\begin{proof}[Sketch of a proof of (a2) and (b2)]\aronly{\footnote{For $l=2$ the equality (a2) alternatively follows because using the definition of the degree and simple properties of linking coefficients, we see that $H((\widehat x_1\vee\widehat x_2)\circ W_2(a))=2ax_1x_2$, and because the Freudenthal-Pontryagin Theorem
states that {\it if $H(\varphi)=H(\psi)$ for maps $\varphi,\psi:S^3\to S^2$, then $\varphi\simeq\psi$}.}}
For a complex $X$ and maps $f,g:S^l\to X$ define a map $[f,g]:S^{2l-1}\to X$ in the same way as $w$ except that $S^2\vee S^2,\inc,\incj$ are replaced by $X,f,g$.
Then $w=[\inc,\incj]$, $W_2(a)=[\inc\circ\widehat a,\incj]$, and $W(b)=[\widehat b,\widehat 1]$.
This construction defines a map $[\cdot,\cdot]:\pi_l(X)\times \pi_l(X)\to\pi_{2l-1}(X)$ (called {\it Whitehead product}).
For $l>1$ we have
$$[\alpha_1+\alpha_2,\gamma]=[\alpha_1,\gamma]+[\alpha_2,\gamma]\quad\text{and}\quad [\alpha,\gamma]=(-1)^l[\gamma,\alpha].$$
Then
$$(\widehat x_1\vee\widehat x_2)\circ W_2(a) \simeq [\widehat x_1\circ\widehat a,\widehat x_2] \simeq W(ax_1x_2).$$
For $l$ odd denoting the homotopy class of a map by the same letter as the map we have
$$(\inc\circ\widehat x+\incj\circ\widehat y)\circ W_2(a) \overset{(1)} = a[x_1\inc+y_1\incj,x_2\inc+y_2\incj] \overset{(2)} =
a[x_1\inc,y_2\incj] + a[y_1\incj,x_2\inc] \overset{(3)}= W_2(a(x_1y_2-x_2y_1)).$$
Here
equalities (1), (2), (3) hold because $(\inc\circ\widehat x+\incj\circ\widehat y)|_{S^l_j}=x_j\inc+y_j\incj$, because $[\inc,\inc]=-[\inc,\inc]$,
$[\incj,\incj]=-[\incj,\incj]$, and $a$ is even,\aronly{\footnote{For $l\in\{3,7\}$ the equality (b2) holds for $2a$ replaced by $a$,
because $W(1)=[\widehat 1,\widehat 1]$ is null-homotopic.}}
and because $[\incj,\inc]=-[\inc,\incj]$, respectively.
\end{proof}

\comment


Here is an alternative proof that $(\widehat x_1\vee\widehat x_2)\circ W_2(a) \simeq x_1x_2W(a)$ for $l$ odd and $a$ even.
Denote $f:=(\widehat x\vee \widehat y)\circ W_2(a)$.
Using the definition of the degree and simple properties of linking coefficients, we see that
$H_\vee(f)=a(x_1y_2-x_2y_1)$.
A map $\psi:S^{2l-1}\to S^l\vee S^l$ is {\it Borromean} if its composition with each of the contractions
$S^l\vee S^l\to S^l\vee*$ and $S^l\vee S^l\to *\vee S^l$ is homotopic to a constant map.
Clearly, $W_2(a)$ is Borromean.
It is known that $W(2)$ is null-homotopic for $l$ odd.
The composition of $f$ with the contraction $S^l\vee S^l\to S^l\vee*$ is $x_1x_2W(a)$.
Hence $f$ is Borromean for $a$ even.
The Hilton Theorem on homotopy classification of maps $S^{2l-1}\to S^l\vee S^l$ implies that
{\it if $H_\vee(\varphi)=H_\vee(\psi)$ for Borromean maps $\varphi,\psi:S^{2l-1}\to S^l\vee S^l$, then $\varphi\simeq\psi$} (this corollary was presumably proved earlier by Whitehead).
Hence $f\simeq W_2(a(x_1y_2-x_2y_1))$ for $a$ even.


$$\xymatrix{A \ar[d]_{\alpha} \ar[r]^{\subset} \ar[dr]^{\beta} & \Cyl\alpha \ar[r]^{\subset} \ar@{-->}[d] & \Cyl(\alpha,\beta) \ar@{-->}[dl]  \\
V \ar@{-->}[r]_{\varkappa} &  Y \ar@{=}[r] & Y \ar[u]_{\subset} }.$$

\begin{proposition}\label{p:cyl}
For any complexes $A,V,Y$ and PL maps $\alpha:A\to V$, $\beta:A\to Y$ the following properties are equivalent:

(1) there is a PL map $\varkappa:V\to Y$ such that $\beta\simeq\varkappa\circ\alpha$;

(2) the map $\beta$ extends to a PL map $\Cyl\alpha\to Y$;

(3) the identity map of $Y$ extends to a PL map $\Cyl(\alpha,\beta)\to Y$, where the `double mapping cylinder'
$\Cyl(\alpha,\beta)$ is the union of $\Cyl\alpha$ and $\Cyl\beta\supset Y$, in which $A\subset\Cyl\alpha$ is  identified with $A\subset\Cyl\beta$.
\end{proposition}

\begin{proof}[Proof of Proposition \ref{p:cyl}]
$(3)\Rightarrow(1)$ (or $(2)\Rightarrow(1)$).
Let $\varkappa$ be the restriction to $V\subset \Cyl\alpha$ of given extension.

$(1)\Rightarrow(2)$.
Define the map $\ret\alpha:\Cyl\alpha\to Q$ by mapping to $\alpha(u)$ the segment containing $(u,0,0)$ from the definition of $\Cyl\alpha$.
By the Borsuk Homotopy Extension Theorem extendability is equivalent to homotopy extendability.
So we can take the required extension to be $\varkappa\circ\ret\alpha$.

$(2)\Rightarrow(3)$. Define the required extension to be $\ret\beta$ on $\Cyl\beta$
and to be the given extension on $\Cyl\alpha$.
\end{proof}




\endcomment

\section{Appendix: is embeddability of complexes undecidable in codimension $>1$?}\label{s:plan}

\UseRawInputEncoding

Realizability of hypergraphs or complexes in the $d$-dimensional Euclidean space $\R^d$ is defined similarly to the realizability of graphs in the plane.
E.g. for 2-complex one `draws' a triangle for every three-element subset.
There are different formalizations of the idea of realizability.


A complex $(V,F)$ is {\bf simplicially} (or linearly) {\bf embeddable} in $\R^d$ if there is a set $V'$ of distinct points in $\R^d$ corresponding to $V$ such that for any subsets $\sigma,\tau\subset V'$ corresponding to elements of $F$ the convex hull $\left<\sigma\right>$ is a
simplex of dimension $|\sigma|-1$ and $\left<\sigma\right>\cap\left<\tau\right>=\left<\sigma\cap\tau\right>$.
\algor{\footnote{This property means that there is an embedded set of simplices in $\R^d$ whose vertices correspond to $V$ and whose simplices correspond to $F$ (an embedded set of simplices of different dimensions in $\R^d$ is defined analogously to \S\ref{0-reaemb}).}
Это свойство формализует <<отсутствие самопересечений>>.}

A complex is {\bf PL} (piecewise linearly) {\bf embeddable} in $\R^d$ if some its subdivision is simplicially embeddable in $\R^d$.\aronly{\footnote{The
related different notion of being topologically embeddable is not required in this text.}}

For classical and modern results on embeddability and their discussion see e.g. surveys \cite{Sk06}, \cite[\S3]{Sk18}, \cite[\S5]{Sk}.

\begin{theorem}[embeddability is undecidable in codimension 1]\label{t:undec1} For every fixed $d,k$ such that
$5\le d\in\{k,k+1\}$ there is no algorithm recognizing PL embeddability of $k$-complexes in $\R^d$.
\end{theorem}

This is deduced in \cite[Theorem 1.1]{MTW} from the Novikov theorem on unrecognizability of the $d$-sphere.
Cf. \cite[Remark 3]{NW97}.

\begin{conjecture}[embeddability is undecidable in codimension $>1$]\label{t:undec} For every fixed $d,k$ such that $8\le d\le \frac{3k+1}2$ there is no algorithm recognizing PL embeddability of $k$-complexes in $\R^d$.
\end{conjecture}

\aronly{Conjecture \ref{t:undec} easily follows from its `extreme' case $2d=3k+1=6l+4$ \cite[Corollaries 4 and 6]{FWZ}.
The extreme case is implied by the equivalence $(SKEW)\Leftrightarrow(Em)$ of Conjecture \ref{t:ckmvwzg} below.\footnote{The extreme case is also implied by the equivalence between $(SKEW1)$ of Conjecture \ref{c:solodd}.a and the analogue of $(Em2)$ from Conjecture \ref{t:ckmvwzg} for `almost embedding' replaced by
`embedding'.
The extreme case for $l$ even is also implied by the equivalence between $(SYM1)$ of Conjecture \ref{c:solodd}.b and the analogue of $(Em1)$ from Conjecture \ref{t:ckmvwzp} for `almost embedding' replaced by
`embedding'.}}

Conjecture \ref{t:undec} is stated as a theorem in \cite{FWZ}.
The proof in \cite{FWZ} contains a gap described below.
Their idea is to elaborate the following remark to produce the reduction (described below) to the `retractability is undecidable' Theorem \ref{t:hopfwhs}.


\begin{remark}
Homotopy classifications of maps $S^{2l-1}\to S^l$ and $S^{2l-1}\to S^l\vee S^l$ are related to isotopy classification of links of $S^{2l-1}\sqcup S^{2l-1}$ and of $S^{2l-1}\sqcup S^{2l-1}\sqcup S^{2l-1}$ in $\R^{3l}$ \cite{Ha62l}
(including higher-dimensional Whitehead link and Borromean rings \cite[\S3]{Sk06}).
E.g. the {\it generalized linking coefficients} of the Whitehead link and of the Borromean rings are (the homotopy classes)
of the Whitehead maps $W(1):S^{2l-1}\to S^l$ and $W_2(1):S^{2l-1}\to S^l\vee S^l$ from Theorem \ref{t:himi}.
Analogous results for $l=1$ do illustrate some ideas, see a description accessible to non-specialists in \cite[\S3.2]{Sk20}.
\end{remark}



We use the notation of \S\ref{s:mainre}.
In this section $a=((a^{i,j})_1,\ldots,(a^{i,j})_m)$, $1\le i<j\le s$, and $b=(b_1,\ldots,b_m)$ are arrays of integers.
Define the {\it double mapping cylinder} $X(a,b)$
to be the union of $\Cyl W_s(a)$ and $\Cyl W_2(b)\supset Y$, in which $V_m^{2l-1}\subset\Cyl W_s(a)$ is  identified with $V_m^{2l-1}\subset\Cyl W_2(b)$.

Assume that $S^{2l+1}\vee S^{2l+1}$ is standardly embedded into $S^{3l+2}$.
Take a small oriented $(l+1)$-disks $D_+,D_-\subset S^{3l+2}$

$\bullet$ intersecting at a point in $\partial D_+\cup\partial D_-$;

$\bullet$ whose intersections with $S^{2l+1}\vee S^{2l+1}$ are transversal and consist of exactly one point $D_+\cap(S^{2l+1}\vee S^{2l+1})\in S^{2l+1}\vee*$ and $D_-\cap(S^{2l+1}\vee S^{2l+1})\in *\vee S^{2l+1}$.

Define {\it the meridian $\Sigma^l\vee\Sigma^l$ of $S^{2l+1}\vee S^{2l+1}$ in $S^{3l+2}$} to be
$\partial D_+\cup\partial D_-$.

\begin{conjecture}\label{t:ckmvwz} For any odd integer $l$ and all $a^{i,j}_q$ even there is a $(2l+1)$-complex $G\supset S^l\vee S^l$ such that any of the following properties is equivalent to (SKEW):

(Ex) a PL homeomorphism of $S^l\vee S^l\to\Sigma^l\vee\Sigma^l$ of $S^{2l+1}\vee S^{2l+1}$ in $S^{3l+2}$ extends to a PL map $X(a,b)\to S^{3l+2}-(S^{2l+1}\vee S^{2l+1})$.

(Ex') a PL homeomorphism of $S^l\vee S^l\to\Sigma^l\vee\Sigma^l$ extends to a PL embedding
$X(a,b)\to S^{3l+2}-(S^{2l+1}\vee S^{2l+1})$.

(Em) $X(a,b)\cup_{S^l\vee S^l}G$ embeds into $S^{3l+2}$.
\end{conjecture}

All the implications except $(Em)\Rightarrow(Ex')$ are correct results of \cite{FWZ}.

The implication $(Ex')\Rightarrow(Ex)$ is clear.

The equivalence of $(Ex)$ and (SKEW)
follows by Propositions \ref{p:ld} and \ref{p:cylnew}.ab because there is a {\it strong deformation retraction}
$S^{3l+2}-(S^{2l+1}\vee S^{2l+1})\to\Sigma^l\vee \Sigma^l$.


The implication $(Ex)\Rightarrow(Ex')$ is implied by the following version of {\it the Zeeman-Irwin Theorem}
\cite[Theorem 2.9]{Sk06}.


\begin{lemma}\label{t:irwstr}
For any PL map $f:X(a,b)\to S^{3l+2}-(S^{2l+1}\vee S^{2l+1})$
there is a PL embedding $f':X(a,b)\to S^{3l+2}-(S^{2l+1}\vee S^{2l+1})$ such that the restrictions of $f$ and $f'$ to $S^l\vee S^l\subset X(a,b)$ are homotopic.
\end{lemma}

\jonly{Lemma \ref{t:irwstr} is essentially a restatement of \cite[Theorem 10]{FWZ} accessible to non-specialists. See more detailed historical remark in \cite[Remark 3.9]{Sk20e}.}


The idea of \cite{FWZ} to prove the implication $(Em)\Rightarrow(Ex')$ is to construct the complex $G$, and use a modification of the following Lemma \ref{l:ramsey}.


\begin{lemma}[{\cite[Lemma 1.4]{SS92}}]\label{l:ramsey}
For any integers $0\le l<k$ there is a $k$-complex $F_-$ containing subcomplexes $\Sigma^k\cong S^k$ and $\Sigma^l\cong S^l$, PL embeddable into $\R^{k+l+1}$ and such that for any PL
embedding $f:F_-\to\R^{k+l+1}$ the images $f\Sigma^k$ and $f\Sigma^l$ are linked modulo 2.
\end{lemma}

Lemma 30 of \cite{FWZ} is a modification of Lemma \ref{l:ramsey} with `linked modulo 2' replaced by `linked with linking coefficient $\pm1$'.
The proof of \cite[p. 778, end of proof of Lemma 30]{FWZ} used the following incorrect statement: {\it If $f:D^p\to\R^{p+q}$ and $g:S^q\to\R^{p+q}$ are PL embeddings such that $|f(D^p)\cap g(D^q)|=1$, then the linking coefficient of $f|_{S^{p-1}}$ and $g$ is $\pm1$.}

\begin{example}\label{e:link}
For any integers $p,q\ge2$ and $c$ there are PL embeddings $f:D^p\to\R^{p+q}$ and $g:S^q\to\R^{p+q}$ such that
$|f(D^p)\cap g(S^q)|=1$ and the linking coefficient of $f|_{S^{p-1}}$ and $g$ is $c$.
\end{example}

\begin{proof} Take PL embeddings $f_0:S^{p-1}\to\R^{p+q-1}$ and $g_0:S^{q-1}\to\R^{p+q-1}$ whose linking coefficient is $c$.
Take points $A,B\in\R^{p+q}-\R^{p+q-1}$ on both sides of $\R^{p+q-1}$.
Then $f=f_0*A$ and $g=g_0*\{A,B\}$ are the required embeddings.
\end{proof}

The modification \cite[Lemma 30]{FWZ} of
Lemma \ref{l:ramsey} is presumably incorrect\jonly{, cf. \cite[Theorem 1.6]{KS20}. See more discussion and conjectures in \cite[\S3]{Sk20e}}.

\aronly{

\begin{theorem}[{\cite[Theorem 1.6]{KS20}}]\label{t:ss1ae}
For any integers $1<l<k$ and $z$ there is a PL almost embedding $f:F_-\to\R^{k+l+1}$ such that
the linking coefficients of oriented $f\Sigma^k$ and $f\Sigma^l$ is $2z+1$.
\end{theorem}

\begin{remark}\label{r:huze} (a) Lemma \ref{t:irwstr} is essentially a restatement of \cite[Theorem 10]{FWZ} accessible to non-specialists.
Analogous lemma for $X(a,b)$ replaced by $2l$-dimensional $(l-2)$-connected manifold is (a particular case of) the Zeeman-Irwin Theorem.
The required modification of the Zeeman-Irwin proof is not hard.
It is based on a version of engulfing similar to \cite[\S2.3]{Sk98} (such a version was possibly suggested by C.  Zeeman to C. Weber \cite[\S2, the paragraph before remark 1]{We67}).

(b) Proposition 34 of \cite{FWZ} is a detailed general position argument for the following
statement: {\it If $Z$ is a subcomplex of a complex $X$
and $2\dim Z<d$, then any PL map of $X$ to a PL $d$-manifold is homotopic to a PL map the closure of whose self-intersection set misses $Z$.}
(This should be known, at least in folklore, but I do not immediately see a reference.)

(c) Lemma 41 of \cite{FWZ} is a version of the following theorem: {\it Any PL map of $S^n\times I$ to an $(2n+3-m)$-connected $m$-manifold $Q$ is homotopic to a PL embedding}
(this is a particular case of \cite[Theorem 8.3]{Hu69}).
The novelty of \cite[Lemma 41]{FWZ} is the property $S(g_1)\subset S(g)$.
This property is not checked in \cite[proof Lemma 41]{FWZ} but does follow from
$C\cap g(\Cl(A\times[0,1]-\sigma))=g(\widetilde I)$; the latter holds
because of the `metastable dimension restriction' $2(3l+2)\ge 3(2l+1)$.

(d) In the proof of \cite[Lemma 42]{FWZ} the property $S(g_1)\subset S(g)$ is not checked.
This property ensures that we can make new improvements without destroying the older ones.
Cf. \cite[line 5 after the display formula in p. 2468]{Sk98}.
This property presumably holds because of the `metastable dimension restriction' $2(3l+2)\ge 3(2l+1)$.
\end{remark}

A PL map $g:K\to\R^d$ of a complex $K$ is called an {\bf almost embedding} if $g\alpha\cap g\beta=\emptyset$ for any two disjoint simplices $\alpha,\beta\subset K$.

\begin{conjecture}[almost embeddability is undecidable]\label{t:aundec}
For every fixed $d,k$ such that

(a) $5\le d\in\{k,k+1\}$; \quad (b) $8\le d\le \frac{3k+1}2$

there is no algorithm recognizing almost embeddability of $k$-complexes in $\R^d$.
\end{conjecture}



Conjecture \ref{t:aundec} easily follows from its `extreme' case $2d=3k+1=6l+4$ analogously to
\cite[Corollaries 4 and 6]{FWZ}.
The extreme case for $l$ even is implied by the equivalence $(SYM1)\Leftrightarrow(Em1)$ of the following Conjectures \ref{c:solodd}.b and Proposition \ref{t:ckmvwzp}.
The extreme case for any $l$ is implied by the equivalence $(SKEW1)\Leftrightarrow(Em2)$ of the following Conjectures \ref{c:solodd}.a and \ref{t:ckmvwzg}.


\begin{conjecture}
\label{c:solodd} (a) For some fixed integers $m,s$ there is no algorithm which for given arrays $a=(a^{i,j}_q)$, $1\le i<j\le s$, $1\le q\le m$ and $b=(b_1,\ldots,b_m)$ of integers
decides whether

\quad (SKEW1) there are integers $x_1,\ldots,x_s,y_1,\ldots,y_s,z$ such that
$$\sum\limits_{1\le i<j\le s}a^{i,j}_q(x_iy_j-x_jy_i)=(2z+1)b_q,\quad 1\le q\le m.$$

(b) For some fixed integers $m,s$ there is no algorithm which for given arrays $a=(a^{i,j}_q)$, $1\le i<j\le s$, $1\le q\le m$, and $b=(b_1,\ldots,b_m)$ of integers decides whether

\quad (SYM1) there are integers $x_1,\ldots,x_s,z$ such that
$$\sum\limits_{1\le i<j\le s}a^{i,j}_qx_ix_j=(2z+1)b_q,\quad 1\le q\le m.$$
\end{conjecture}

\begin{remark}\label{r:equiv}
B. Moroz conjectured and E. Kogan sketched a proof that Conjecture \ref{c:solodd}.a is equivalent to:

{\it (*) for some fixed positive integers $m,s$ there is no algorithm which for a given system
of $m$ Diophantine equations in $s$ variables decides whether the system has a solution in rational numbers with odd denominators.}

Since $m$ equations are equivalent to 1 equation (sum of squares) and since work of
J. Robinson characterizes the rational numbers with odd denominators among all rational numbers in a Diophantine way, (*) is in turn is equivalent to:

{\it (**) for some fixed positive integer $s$ there is no algorithm which for a given
  polynomial equation with integer coefficients in $s$ variables
  decides whether the system has a solution in rational numbers.}

The statement (**) is an open problem.
\end{remark}

An {\bf odd (almost) embedding} is a PL (almost) embedding $f:S^l\to S^{3l+2}-S^{2l+1}$ such that $f(S^l)$ is linked modulo 2 with $S^{2l+1}$.

\begin{proposition}\label{t:ckmvwzp} For any even $l$ there is a $(2l+1)$-complex $G_1\supset S^l$ such that any of the following properties is equivalent to (SYM1):

(Ex1) some odd almost embedding extends to a PL map of $X(a,b)$.

(Ex'1) some odd almost embedding extends to a PL embedding of $X(a,b)$.

(Em1) $X(a,b)\cup_{S^l}G_1$ embeds into $S^{3l+2}$.
\end{proposition}

All the implications except $(Em1)\Rightarrow(Ex'1)$ (and their analogues for `almost embedding' replaced by
`embedding') are proved analogously to the corresponding correct implications of Conjecture \ref{t:ckmvwz}.
The implication $(Em1)\Rightarrow(Ex'1)$ (and its analogue) follows by Theorem \ref{t:ss1ae}
(by the conjecture in \cite[Remark 1.7.b]{KS20})
analogously to \cite{FWZ}.


An {\bf odd (almost) embedding} is a PL (almost) embedding
$f:S^l_1\vee S^l_2\to S^{3l+2}-S^{2l+1}_1\vee S^{2l+1}_2$ such that the mod 2 linking coefficient of $f(S^l_i)$ and $S^{2l+1}_j$ equals to the Kronecker delta $\delta_{i,j}$.

\begin{conjecture}\label{t:ckmvwzg} For any odd $l>1$ and all $a^{i,j}_q$ even there is a $(2l+1)$-complex $G_2\supset S^l\vee S^l$ such that any of the following properties is equivalent to (SKEW1):

(Ex2) some odd almost embedding extends to a PL map of $X(a,b)$.

(Ex'2) some odd almost embedding extends to a PL embedding of $X(a,b)$.

(Em2) $X(a,b)\cup_{S^l\vee S^l}G_2$ embeds into $S^{3l+2}$.
\end{conjecture}

All the implications except $(Em2)\Rightarrow(Ex'2)$ (and their analogues for `almost embedding' replaced by
`embedding') are proved analogously to the corresponding correct implications of Conjecture \ref{t:ckmvwz}.
The implication $(Em2)\Rightarrow(Ex'2)$ (and its analogue) would follow by a `wedge' analogue of Theorem \ref{t:ss1ae}
(and of the conjecture in \cite[Remark 1.7.b]{KS20})
analogously to \cite{FWZ}.

}

\end{document}